\documentclass[11pt,a4paper]{article}
\usepackage{epsf,epsfig,amsfonts,amsgen,amsmath,amstext,amsbsy,amsopn,amsthm}
\usepackage{amsmath}
\usepackage{enumerate}
\usepackage{hhline}
\usepackage{multirow}
\usepackage{multicol}
\usepackage{booktabs}
\usepackage{lineno}
\usepackage{amsfonts,amsthm,amssymb,bm}
\usepackage{amsfonts}
\usepackage{graphics}
\usepackage{latexsym,bm}
\usepackage{amsfonts,amsthm,amssymb,bbding}
\usepackage{indentfirst}
\usepackage{graphicx}
\usepackage{color}
\usepackage[colorlinks=true,anchorcolor=blue,filecolor=blue,linkcolor=red,urlcolor=blue,citecolor=blue]{hyperref}
\usepackage{float}
\usepackage{tikz,enumerate}
\usetikzlibrary{calc}
\usepackage{geometry}
\newcommand{\n}{\noindent}

\newtheorem{thm}{Theorem}[section]
\newtheorem{conj}[thm]{Conjecture}

\renewcommand{\le}{\leqslant}

\renewcommand{\ge}{\geqslant}

\begin{document}

\title{On two conjectures of Ho\`ang}

\author{
Hongzhang Chen\thanks{School of Mathematics and Statistics, Gansu
Center for Applied Mathematics, Lanzhou University, Lanzhou, Gansu,
730000, China. Email: \url{mnhzchern@gmail.com}.} 
\and 
Kaiyang Lan\thanks{Corresponding author. School of Mathematics and Statistics, Minnan Normal University, Zhangzhou, Fujian, 363000, China. Email: \url{kylan95@126.com}. Partially supported by the Youth Foundation of Fujian Province (Grant No. JZ240035), 
the Minnan Normal University Foundation (Grant No. KJ2023002).} 
\and Wenlong Zhong\thanks{School of Mathematics and Statistics, Minnan Normal University, Zhangzhou, Fujian, 363000, China. Email: \url{2364810512@qq.com}.} }

\date{\today}
\maketitle

% \linenumbers \pagewiselinenumbers

% \modulolinenumbers[2]

\begin{abstract}
A graph $G$ is said to be \textit{perfectly divisible} if for every induced subgraph $H$ of $G$ with at least one edge, the vertex set $V(H)$ can be partitioned into two sets $A, B$ such that $H[A]$ is perfect and $\omega(B) < \omega(H)$.
It is easy to see that the chromatic number of a perfectly divisible graph is at most $\binom{\omega(G)+1}{2}$.
Ho\`ang conjectured that every graph $G$ with $\alpha(G) \le 3$ is perfectly divisible.
We disprove this conjecture.

In the same vein, a graph $G$ with at least one edge is \textit{$k$-divisible} if for every induced subgraph $H$ of $G$ with at least one edge, the vertex set $V(H)$ can be partitioned into $k$ sets, none of which contains a largest clique of $H$.
It is easy to see that the chromatic number of a $k$-divisible graph is at most $k^{\omega-1}$.
Ho\`ang conjectured that every even-hole-free graph is 3-divisible.
We confirm this conjecture.
\end{abstract}



\section{Introduction}
All graphs considered in this paper are finite and simple. For a positive integer $k$, we write $[k] = \{1,2,\ldots,k\}$. 
Let $G$ be a graph. A $k$-\emph{coloring} of $G$ is a mapping $\varphi: V(G) \to [k]$ with $\varphi(u) \neq \varphi(v)$ whenever $uv \in E(G)$.
The \emph{chromatic number} $\chi(G)$ is the smallest $k$ for which such a coloring exists.
For a vertex set $X \subseteq V(G)$, denote by $G[X]$ the subgraph induced by $X$; $X$ is called a \emph{clique} if $G[X]$ is complete.
The \emph{clique number} $\omega(G)$ is the maximum size of a clique in $G$.
An \emph{independent set} in $G$ is a set of pairwise non-adjacent vertices.
The \emph{independence number} $\alpha(G)$ is the maximum size of an independent set in $G$.
We denote by $P_k$ and $C_k$ the path and the cycle on $k$ vertices, respectively.

\smallskip
We say that $G$ \emph{contains} a graph $H$ if some induced subgraph of $G$ is isomorphic to $H$. A graph is $H$-\emph{free} if it does not contain $H$, and for a set $\mathcal{H}$ of graphs, $G$ is $\mathcal{H}$-\emph{free} if it contains no member of $\mathcal{H}$. A class $\mathcal{C}$ of graphs is \emph{hereditary} if it is closed under taking induced subgraphs. A hereditary class $\mathcal{G}$ is called $\chi$-\emph{bounded} if there exists a function $f$ (a $\chi$-\emph{binding function}) such that $\chi(G) \le f(\omega(G))$ for all $G \in \mathcal{G}$. If $f$ can be chosen to be a polynomial, then $\mathcal{G}$ is said to be \emph{polynomially $\chi$-bounded}. We refer the reader to \cite{BRIS2004,ISBR2019,AS20} for further results in this area.

\smallskip
A graph $G$ is \emph{perfect} if $\chi(H) = \omega(H)$ for every induced subgraph $H$ of $G$, and \emph{imperfect} otherwise.
Following \cite{CTH2018}, a graph $G$ is \emph{perfectly divisible} if for every induced subgraph $H$ of $G$ with at least one edge, the vertex set $V(H)$ admits a partition into two parts $A$ and $B$ such that $H[A]$ is perfect and $\omega(H[B]) < \omega(H)$. A simple induction on $\omega(G)$ then shows that every perfectly divisible graph $G$ satisfies $\chi(G) \le \binom{\omega(G)+1}{2}$ (see \cite{MC2019}). Consequently, the class of perfectly divisible graphs is polynomially $\chi$-bounded.
For further background and recent results on perfect divisibility, see \cite{RC2023,RCxz2026,MC2021,MC2019,CTH2018,WD2022,TKJK2022,DW2023}.

\smallskip
A graph $G$ is \emph{minimally non-perfectly divisible} if $G$ is not perfectly divisible but every proper induced subgraph of $G$ is perfectly divisible. A set $C$ of vertices of $G$ is a \emph{clique cutset} if $C$ induces a clique in $G$ and $G \setminus C$ is disconnected.
Ho\`ang proved that every minimally non-perfectly divisible graph $G$ with $\alpha(G) \le 3$ cannot contain a clique cutset and proposed the following conjecture.

\begin{conj}[\cite{Hoang2026}]\label{con4.3}
	Every graph $G$ with $\alpha(G) \le 3$ is perfectly divisible.
\end{conj}

Our first result shows that this conjecture is false.

\begin{thm}\label{counterexample4.3}
	There exists a graph $G$ with $\alpha(G) \le 3$ that is not perfectly divisible.
\end{thm}

Let $k \ge 2$ be an integer.
A graph $G$ with at least one edge is \emph{$k$-divisible} if for every induced subgraph $H$ of $G$ with at least one edge, the vertex set $V(H)$ can be partitioned into $k$ sets, none of which contains a maximum clique of $H$.
As shown in \cite{MC2019,Hoang2026}, the class of $k$-divisible graphs admits an exponential $\chi$-bounding function.

\smallskip
A \emph{hole} is a chordless cycle of length at least four.
A hole is \emph{even} if its length is even, and \emph{odd} otherwise.
Ho\`ang proposed the following conjecture.

\begin{conj}[\cite{Hoang2026}]\label{con4.8}
	Every even-hole-free graph is $3$-divisible.
\end{conj}

Our second result shows that this conjecture is true.
Indeed, we confirm it by establishing a more general statement.

\smallskip
A vertex $v$ in a graph $G$ is called \emph{$k$-simplicial} if its neighborhood $N_G(v)$ can be partitioned into $k$ (possibly empty) cliques.
(Note that a $1$-simplicial vertex is precisely a simplicial vertex, and a $2$-simplicial vertex a bisimplicial vertex.)

\begin{thm}\label{th4.8}
	Let $k \ge 2$ be an integer.
	Let $\mathcal{G}$ be a hereditary class of graphs such that every graph $G \in \mathcal{G}$ contains a $k$-simplicial vertex. 
	Then every graph in $\mathcal{G}$ is $(k+1)$-divisible.
\end{thm}

\begin{proof}[\textbf{Proof of Conjecture~\ref{con4.8}, assuming Theorem~\ref{th4.8}.}]
	Let $G$ be an even-hole-free graph with at least one edge.
	By a classical result of Chudnovsky and Seymour~\cite{CS2023}, every even-hole-free graph contains a $2$-simplicial vertex, that is, a vertex whose neighborhood can be partitioned into two cliques.
	Moreover, the class of even-hole-free graphs is hereditary.
	Taking $k = 2$ in Theorem~\ref{th4.8}, we conclude that every even-hole-free graph is $3$-divisible.
	This proves Conjecture~\ref{con4.8}.
\end{proof}

\section{Proofs}

The main idea of the proof of Theorem~\ref{counterexample4.3} is to use Ramsey numbers and the fact that perfectly divisible graphs satisfy $\chi(G) \le \omega(G)^2$ \cite{Hoang2026}.

\begin{proof}[\bf Proof of Theorem~\ref{counterexample4.3}.]
	For positive integer $t$, let $R(4,t)$ be the minimum positive integer $N$ such that 
	every graph on $N$ vertices contains a clique of size $4$ or an independent set of size $t$. Ramsey's theorem \cite{Ramsey1930} states that $R(4,t)$ exists.
	By the asymptotic bound of Mattheus and Verstraete \cite{MV2024}, there exists a constant $c>0$ such that
	\[
	R(4,t) \ge c \cdot \frac{t^3}{\log^4 t}.
	\]
	In particular, for sufficiently large $t$,
	\[
	R(4,t) - 1 > 3(t-1)^2.
	\]
	Fix such a $t$. By definition of $R(4,t)$, there exists a graph $H$ on $|V(H)| = R(4,t)-1$ vertices with $\omega(H) \le 3$ and $\alpha(H) \le t-1$.
	
	Let $G = \overline{H}$ be the complement of $H$. Then
	\[
	\alpha(G) = \omega(H) \le 3, \quad
	\omega(G) = \alpha(H) \le t-1, \quad
	|V(G)| = R(4,t)-1.
	\]
	Since $\alpha(G) \le 3$, we have
	\[
	\chi(G) \ge \frac{|V(G)|}{\alpha(G)} \ge \frac{R(4,t)-1}{3}.
	\]
	If $G$ were perfectly divisible, then by the bound of Ho\`ang~\cite{Hoang2026} we have $\chi(G) \le \omega(G)^2 \le (t-1)^2$. 
	But our choice of $t$ guarantees
	\[
	\frac{R(4,t)-1}{3} > (t-1)^2,
	\]
	a contradiction. Hence $G$ is not perfectly divisible, while $\alpha(G) \le 3$.	
	This proves Theorem~\ref{counterexample4.3}.
\end{proof}

Next, we prove Theorem~\ref{th4.8}.

\begin{proof}[\bf Proof of Theorem~\ref{th4.8}.]
	Let \(G \in \mathcal{G}\) be a graph with at least one edge.
	We prove the theorem by induction on \(|V(G)|\).
	If \(|V(G)| = 1\) or \(G\) has no edges, the statement is vacuously true.
	Assume \(|V(G)| \ge 2\) and \(G\) has at least one edge.
	By hypothesis, there exists a vertex \(v \in V(G)\) such that
	\(N_G(v) \subseteq C_1 \cup C_2 \cup \cdots \cup C_k\)
	with each \(C_i\) a clique.
	Let \(w = \omega(G)\), \(G' = G \setminus v\), and \(w' = \omega(G')\).
	Clearly, \(w' \in \{w, w-1\}\).
	
	If \(w' = w-1\), then by the induction hypothesis, \(G'\) admits a partition
	\(X_1, X_2, \ldots, X_{k+1}\) such that no \(X_i\) contains a clique of size \(w' = w-1\).
	Add \(v\) to any part, say \(X_1\).
	\begin{itemize}
		\item Any \(w\)-clique containing \(v\) would require a \((w-1)\)-clique in \(N(v) \cap X_1\). 
		Such a \((w-1)\)-clique would be a clique of size \(w-1\) in \(G'\), contradicting the induction hypothesis.
		\item Any \(w\)-clique not containing \(v\) would be a \(w\)-clique in \(G'\), which does not exist because \(\omega(G') = w-1\).
	\end{itemize}
	Thus the partition remains valid for \(G\).
	
	If \(w' = w\), then by the induction hypothesis, \(G'\) has a partition
	\(X_1, X_2, \ldots, X_{k+1}\) such that no \(X_i\) contains a \(w\)-clique.
	Call a part \(X_i\) \emph{bad} if \(X_i \cap N(v)\) contains a \((w-1)\)-clique.
	Since \(C_i \cup \{v\}\) is a clique in \(G\) for each \(i\), we have \(|C_i| \le w-1\).
	Hence,
	\[
	|N(v)| \le |C_1| + |C_2| + \cdots + |C_k| \le k(w-1).
	\]
	
	If all \(k+1\) parts \(X_1, X_2, \ldots, X_{k+1}\) were bad, then \(N(v)\) would contain \(k+1\) pairwise disjoint \((w-1)\)-cliques.
	Consequently, \(|N(v)| \ge (k+1)(w-1)\).
	For \(w \ge 2\) (the case \(w=1\) is trivial), we have \((k+1)(w-1) > k(w-1)\), a contradiction.
	Therefore, at least one part, say \(X_1\), is not bad.
	Now construct a new partition of \(V(G)\) as
	\(X_1' = X_1 \cup \{v\}\) and \(X_i' = X_i\) for \(i = 2, 3, \ldots, k+1\).
	\begin{itemize}
		\item The part \(X_1'\) cannot contain a \(w\)-clique that includes \(v\), otherwise \(X_1 \cap N(v)\) would contain a \((w-1)\)-clique, making \(X_1\) bad.
		\item For \(i \ge 2\), the part \(X_i' = X_i\) still contains no \(w\)-clique by the induction hypothesis.
	\end{itemize}
	Thus \(X_1', X_2', \ldots, X_{k+1}'\) is a valid \((k+1)\)-divisible partition for \(G\).
	
	In both cases, we obtain a partition of \(V(G)\) into \(k+1\) sets, none of which contains a maximum clique of \(G\).
	By induction, every graph in \(\mathcal{G}\) is \((k+1)\)-divisible.
	This proves Theorem~\ref{th4.8}.
\end{proof}

\vspace{6mm}

\n{\bf Acknowledgements:} %The authors would like to thank the
%anonymous referees for their constructive corrections and valuable
%comments on this paper, which have considerably improved the
%presentation of this paper.
This work was supported by the Fujian Key Laboratory of Granular Computing and Applications.

\paragraph{Data availability.}
Data sharing is not applicable to this article as no datasets were generated or
analysed during the current study.

\paragraph{Conflict of interest.}
The authors have no relevant financial or non-financial interests to disclose.


\begin{thebibliography}{99}
    
      \bibitem{RC2023} R. Chen and B. Xu,
      Structure and coloring of ($P_7, C_5$, \text{diamond})-free graphs, \emph{Discrete Appl. Math.} \textbf{372} (2025), 298--307.
      
  
      \bibitem{RCxz2026} R. Chen and X. Zhang,
      Perfect divisibility and coloring in fork-free graphs,
      \textit{Discrete Math.} \textbf{349}(8) (2026), 115146.
      

      \bibitem{MC2021} M. Chudnovsky, S. Huang, T. Karthick and J. Kaufmann,
      Square-free graphs with no induced fork,
      \emph{Electron. J. Combin.} \textbf{28} (2021), \#P2.20.
      
      
      \bibitem{CS2023}
      M. Chudnovsky and P. Seymour,
      Even-hole-free graphs still have bisimplicial vertices,
      \textit{J. Combin. Theory, Ser. B} \textbf{161} (2023), 331--381.
      
      \bibitem{MC2019} M. Chudnovsky and V. Sivaraman,
      Perfect divisibility and 2-divisibility, \emph{J. Graph Theory} \textbf{90} (2019), 54--60.
      
      \bibitem{WD2022} W. Dong, B. Xu and Y. Xu,
      On the chromatic number of some $P_5$-free graphs,
      \emph{Discrete Math.} \textbf{345} (2022), 113004.
      
      \bibitem{CTH2018} C.T. Ho\`ang,
      On the structure of (banner, odd hole)-free graphs, \emph{J. Graph Theory} \textbf{89} (2018), 395--412.
      
      \bibitem{Hoang2026} C.T. Ho\`ang, On the structure of perfectly divisible graphs, \textit{Discrete Math.} \textbf{349}(2) (2026), 114809.
      
      \bibitem{TKJK2022} T. Karthick, J. Kaufmann and V. Sivaraman,
      Coloring graph classes with no induced fork via perfect divisibility,
      \emph{Electron. J. Combin.} \textbf{29} (2022), \#P3.19.
      

      \bibitem{MV2024}
      S. Mattheus and J. Verstraete,
      The asymptotics of $r(4,t)$,
      \textit{Ann. of Math. (2)} \textbf{199} (2) (2024), 919--941.
      
      \bibitem{Ramsey1930}
      F. P. Ramsey,
      On a problem of formal logic,
      \textit{Proc. Lond. Math. Soc.} \textbf{s2--30}(1) (1930), 264--286.
      
    %  \bibitem{BR1993} B. Randerath, The Vizing bound for the chromatic number based on forbidden pairs, PhD thesis, RWTH Aachen, Shaker Verlag, 1993.
      
\bibitem{BRIS2004} B. Randerath and I. Schiermeyer, Vertex colouring and forbidden subgraphs--a survey,
\emph{Graphs Combin.} \textbf{20} (2004), 1--40.
      
\bibitem{ISBR2019} I. Schiermeyer and B. Randerath,
Polynomial $\chi$-binding functions and forbidden induced subgraphs: a survey,
\emph{Graphs Combin.} \textbf{35} (2019), 1--31.
      
\bibitem{AS20} A. Scott and P. Seymour,
A survey of $\chi$-boundedness,
\emph{J. Graph Theory} \textbf{95} (2020), 473--504.
      

\bibitem{DW2023} D. Wu and B. Xu, Perfect divisibility and coloring of some fork-free graphs, \emph{Discrete Math.} \textbf{347} (2024), 114121.

      
\end{thebibliography}
\end{document}